\documentclass[11pt]{article}

\usepackage{lineno,hyperref}

\usepackage{amssymb,amsmath,amsthm}
\usepackage{indentfirst}
\usepackage{graphicx}
\usepackage{epstopdf}
\usepackage{subfigure}
\usepackage{caption}
\usepackage{color}
\usepackage{enumerate}
\usepackage{mathrsfs}
\usepackage{appendix}

\usepackage{titletoc}
\titlecontents{section}[2em]{}{\contentslabel{1.5em}}{}{%
    \titlerule*[5pt]{$\cdot$}\contentspage}
\titlecontents{subsection}[4.2em]{}{\contentslabel{2.3em}}{}{%
    \titlerule*[5pt]{$\cdot$}\contentspage}
\titlecontents{subsubsection}[7.3em]{}{\contentslabel{3.2em}}{}{%
    \titlerule*[5pt]{$\cdot$}\contentspage}

\newtheorem{theorem}{Theorem}[section]
\newtheorem{corollary}[theorem]{Corollary}
\newtheorem{lemma}[theorem]{Lemma}
\newtheorem{proposition}[theorem]{Proposition}
\newtheorem{definition}[theorem]{Definition}

\allowdisplaybreaks[4]

\makeatother

\title{\bf Long-time behavior of a nonlocal dispersal logistic model with seasonal succession}

\date{}

\author{Zhenzhen Li, Binxiang Dai\footnote{Corresponding author. Email address: \url{zhzhli@csu.edu.cn} (Z. Li), \url{bxdai@csu.edu.cn} (B. Dai)}  \\
{\scriptsize School of Mathematics and Statistics, HNP-LAMA, Central South University, Changsha, Hunan, 410083, PR China}}

\begin{document}
\maketitle

\begin{abstract}
This paper is devoted to a nonlocal dispersal logistic model with seasonal succession in one-dimensional bounded habitat, where the seasonal succession accounts for the effect of two different seasons. Firstly, we provide the persistence-extinction criterion for the species, which is different from that for local diffusion model. Then we show the asymptotic profile of the time-periodic positive solution as the species persists in long run.


\end{abstract}

\section{Introduction}
\setcounter{equation}{0}

The nonlocal diffusion as a long range process can well describe some natural phenomena
in many situations (Andreu-Vaillo et al. \cite{AMRT10}, Fife \cite{F03}). Recently, nonlocal diffusion equations have attracted much attention and have been used to simulate different dispersal phenomena in material science (Bates \cite{B06}), neurology (Sun, Yang and Li \cite{SYL14}), population ecology (Hutson et al. \cite{H03}, Kao, Lou and Shen \cite{KLS10}), etc. Especially, the spectral properties of nonlocal dispersal operators and the essential differences between them and local dispersal operators are studied in Coville \cite{C10}, Coville, D\'avila and Mart\'inez \cite{CDM13}, Garc\'ia-Meli\'an and Rossi \cite{GR09},
Shen and Zhang \cite{SZ10} and Sun, Yang and Li \cite{SYL14}. A widely used nonlocal diffusion operator has the form
\[ (J*u-u)(t,x):=\int_{\mathbb R}J(x-y)u(t,y)\mathrm dy-u(t,x), \]
which can capture the factors of \lq long-range dispersal' as well as \lq short-range dispersal'.

Time-varying environmental conditions are important for the growth and survival of species. Seasonal forces in nature are a common cause of environmental change, affecting not only the growth of species but also the composition of communities \cite{DTD09,D88}. The growth of species is actually driven by both external and internal dynamics. For instance, in temperate lakes, phytoplankton and zooplankton grow during the warmer months and may die or lie dormant during the winter. This phenomenon is termed as seasonal succession.

In the present paper, we are concerned with the nonlocal dispersal logistic model with seasonal succession as follows:
\begin{equation}\label{eq1.1}
\begin{cases}
u_t= -\delta u, & i\omega<t\leq (i+\rho)\omega,\ l_1\leq x\leq l_2, \\
\displaystyle u_t= d(J*u-u)(t,x)+u(a-bu), & (i+\rho)\omega<t\leq (i+1)\omega,\ l_1\leq x\leq l_2, \\
u(0,x)=u_0(x), & x\in[l_1,l_2],
\end{cases}
\end{equation}
where $u(t,x)$ is the population density of a species at time $t$ and location $x$ in the one-dimensional bounded habitat $[l_1,l_2]\subset \mathbb R$. All parameters $\delta,a,b$ and $d$ are positive constants. The kernel function $J:\mathbb R\to \mathbb R$ is assumed to satisfy
\begin{enumerate}[${\bf (J):}$]
\item $J\in C(\mathbb R)\cap L^{\infty}(\mathbb R)$ is nonnegative, even, $J(0)>0$ and $\int_{\mathbb R}J(x)\mathrm dx=1$.
\end{enumerate}
Here the parameter $d$ stands for the diffusion rate of the species. Let $J(x-y)$ be the probability distribution of the species jumping from location $y$ to location $x$, then $\int_{\mathbb R}J(x-y)u(t,y)\mathrm dy$ represents the rate where individuals are arriving at location $x$ from all other places and $-u(t,x)=-\int_{\mathbb R}J(x-y)u(t,x)\mathrm dy$ is the rate at which they are leaving location $x$ to travel to all other sites.
In such model, $u=0$ on $\mathbb R\setminus[l_1,l_2]$ represents homogeneous Dirichlet type boundary condition, which implies that the exterior environment is hostile and the individuals will die when they reach the boundary of habitat $[l_1,l_2]$. The initial function $u_0(x)$ is nonnegative continuous function.  Here and in what follows, unless specified otherwise, we always take $i\in \mathbb Z_+=\{0,1,2,\cdots\}$.

In \eqref{eq1.1}, it is assumed that the species $u$ undergoes two different seasons: the bad season and the good season. In the bad season: $i\omega<t< (i+\rho)\omega$, for instance, from winter to spring, the species can notcan not get enough food to feed themselves and its density are declining exponentially. During this season, the population has no ability to move in space. In the good season (for instance, from summer and autumn): $(i+\rho)\omega<t\leq (i+1)\omega$, we assume that the spatiotemporal distribution of the species $u$ are governed by the classical nonlocal dispersal logistic equation. Parameters $\omega$ and $1-\rho$ represent the period of seasonal succession and the duration of the good season, respectively.

In fact, if we define the time-periodic finctions
\begin{equation}\label{eq1.2}\small
D(t)=\begin{cases}
0, & t\in (i\omega,(i+\rho)\omega], \\
d, & t\in ((i+\rho)\omega,(i+1)\omega],
\end{cases}
\bar a(t)=\begin{cases}
-\delta, & t\in (i\omega,(i+\rho)\omega], \\
a, & t\in ((i+\rho)\omega,(i+1)\omega],
\end{cases}
\bar b(t)=\begin{cases}
0, & t\in (i\omega,(i+\rho)\omega], \\
b, & t\in ((i+\rho)\omega,(i+1)\omega],
\end{cases}
\end{equation}
then model \eqref{eq1.1} can be rewritten as
\begin{equation}\label{eq1.3}
\begin{cases}
u_t=D(t)(J*u-u)(t,x)+u(t,x)(\bar a(t)-\bar b(t)u(t,x)),& t>0,l_1\leq x\leq l_2, \\
u(0,x)=u_0(x), & x\in[l_1,l_2],
\end{cases}
\end{equation}
which is a nonlocal dispersal piecewise smooth time-periodic system.

The models with seasonal succession have been investigated by several authors. Ignoring the spatial evolution of the involved species, the effects of seasonal succession on the dynamics of population can be analysed by ODE models, see \cite{HZ12,LK01} and references therein. There are also some investigations on it by the numerical method, see, e.g. \cite{HT95,SSHKL09}. In \cite{HZ12}, Hsu and Zhao first considered the single species model with seasonal succession:
\begin{equation}\label{eq1.2}
\begin{cases}
z_t=-\delta,  & i\omega<t\leq (i+\rho)\omega, \\
z_t=z(a-bz),  & (i+\rho)\omega<t\leq (i+1)\omega, \\
z(0)=z_0\in\mathbb R_+:=[0,\infty), &
\end{cases}
\end{equation}
where $z(t)$ denotes the population density of a species at time $t$. They showed the threshold dynamics of model \eqref{eq1.2}: when $a(1-\rho)-\delta\rho\leq 0$, the unique solution of \eqref{eq1.2} converges to zero among all nonnegative initial value, while when $a(1-\rho)-\delta\rho> 0$, it converges to the unique positive $\omega$-periodic solution of \eqref{eq1.2} for all positive initial value.

Taking spatial factor into account, Peng and Zhao \cite{PZ13} investigated the following local diffusion model with seasonal succession:
\begin{equation}\label{eq1.5}
\begin{cases}
u_t=-\delta u, & i\omega<t\leq (i+\rho)\omega,\ x\in(l_1,l_2), \\
u_t-du_{xx}=u(a-bu),  & (i+\rho)\omega<t\leq (i+1)\omega,\ x\in(l_1,l_2), \\
u(t,l_1)=u(t,l_2)=0, & t\geq 0, \\
u(0,x)=u_0(x)\geq 0, & x\in(l_1,l_2),
\end{cases}
\end{equation}
where the parameter $d$ stands for the intensity of random diffusion. The positive constants $\omega,\rho,\delta,a,b$ have the same biological interpretations as in \eqref{eq1.1}, and the initial function $u_0\in C^2([l_1,l_2])$. Denote by $\lambda_1^l$ the principal eigenvalue of the eigenvalue problem
\[\begin{cases}
\varphi_t=-\delta \varphi+\lambda\varphi, & i\omega<t\leq (i+\rho)\omega,\ x\in(l_1,l_2), \\
\varphi_t-d\varphi_{xx}=a\varphi+\lambda \varphi,  & (i+\rho)\omega<t\leq (i+1)\omega,\ x\in(l_1,l_2), \\
\varphi>0, & (i+\rho)\omega<t\leq (i+1)\omega,\ x\in(l_1,l_2), \\
\varphi(t,l_1)=\varphi(t,l_2)=0, & t\geq 0, \\
\varphi(t,x)=\varphi(t+\omega,x), & x\in(l_1,l_2).
\end{cases}\]
One can calculate exactly that $\lambda_1^l=(1-\rho)(\frac{\pi^2d}{(l_2-l_1)^2}-a)+\rho\delta$. By the consequence of \cite[Theorem 2.3.4]{Z03}, Peng and Zhao \cite{PZ13} has showed that, the solution of \eqref{eq1.5} converges to zero among all nonnegative initial value if $\lambda_1^l\geq 0$, while when $\lambda_1^l<0$, it converges to the unique positive $\omega$-periodic solution of \eqref{eq1.5} for all nonnegative and not identically zero initial value. Specially, we can observe that
\begin{enumerate}[(i)]
\item if $(1-\rho)a-\rho\delta>0$, then the solution of \eqref{eq1.5} converges to the unique positive $\omega$-periodic solution of \eqref{eq1.5} for all nonnegative and not identically zero initial value;
\item if $(1-\rho)a-\rho\delta<0$, then there exists a critical value $\hat l$ such that the solution of \eqref{eq1.5} converges to the unique positive $\omega$-periodic solution of \eqref{eq1.5} for all nonnegative and not identically zero initial value if and only if $l_2-l_1>\hat l$.
\end{enumerate}

The dynamics of the time-periodic nonlocal dispersal logistic equation have been studied by many authors (see \cite{RS12,SLW17,SV19,SLLY20}). In \cite{RS12}, Rawal and Shen studied the eigenvalue problems of time-periodic nonlocal dispersal operator, and then showed that the existence of positive periodic solution relies on the sign of principal eigenvalue of a linearized eigenvalue problem. Sun et al. \cite{SLW17} considered a time-periodic nonlocal dispersal logistic equation in spatial degenerate environment. Shen and Vo \cite{SV19} and Su et al. \cite{SLLY20} have studied the asymptotic profiles of the generalised principal eigenvalue of time-periodic nonlocal dispersal operators under Dirichlet type boundary conditions and Neumann type boundary conditions, respectively. The models considered in the above mentioned work are all smooth periodic systems.

The purpose of current paper is to study the dynamical properties of nonlocal dispersal model  \eqref{eq1.1}. Clearly, system \eqref{eq1.1} is in time-periodic environment and the dispersal term and reaction term are both discontinuous and periodic in $t$ caused by the seasonal succession. Note that, by general semigroup theory (see \cite{P83}), \eqref{eq1.1} has a unique local solution $u(t,\cdot;u_0)$ with initial value $u(0,\cdot;u_0)=u_0\in C([l_1,l_2])$, which is continuous in $t$.
If $u_0$ is nonnegative over $[l_1,l_2]$, then by a comparison argument, $u(t,\cdot;u_0)$ exists and is nonnegative for all $t>0$ (see Lemma \ref{lem3.2}). Next, we have the following theorem on the long time behavior of model \eqref{eq1.1}.
\begin{theorem}\label{thm1.1}
Assume that ${\bf (J)}$ holds and $-\infty<l_1<l_2<+\infty$. Let $u(t,\cdot;u_0)$ be the unique solution to \eqref{eq1.1} with the initial value $u_0(x)\in C([l_1,l_2])$, where $u_0(x)$ is nonnegative and not identically zero. Then the following statements are true:
\begin{enumerate}[(1)]
\item If $(1-\rho)a-\rho\delta>(1-\rho)d$, then $\lim\limits_{n\to\infty}u(t+n\omega,x;u_0)=u_{(l_1,l_2)}^*(t,x)$ in $C([0,\omega]\times [l_1,l_2])$, where $u_{(l_1,l_2)}^*(t,x)$ is the unique $\omega$-periodic positive solution of
    \begin{equation}\label{eq1.6}
    \begin{cases}
    u_t= -\delta u, & i\omega<t\leq (i+\rho)\omega,\ l_1\leq x\leq l_2, \\
    \displaystyle u_t= d\int_{l_1}^{l_2}J(x-y)u(t,y)\mathrm dy-du(t,x)+u(a-bu), & (i+\rho)\omega<t\leq (i+1)\omega,\ l_1\leq x\leq l_2, \\
    u(t,x)=u(t+\omega,x), & t\geq 0,\ x\in[l_1,l_2];
    \end{cases}
    \end{equation}
\item If $0<(1-\rho)a-\rho\delta\leq (1-\rho)d$, then there exists a unique $\ell^*>0$ such that $\lim\limits_{n\to\infty}u(t+n\omega,x;u_0)=u_{(l_1,l_2)}^*(t,x)$ in $C([0,\omega]\times [l_1,l_2])$ if and only if $l_2-l_1>\ell^*$;
\item If $(1-\rho)a-\rho\delta\leq 0$, then $0$ is the unique nonnegative solution of \eqref{eq1.6}, and $\lim\limits_{t\to\infty}u(t,x;u_0)=0$ uniformly for $x\in[l_1,l_2]$.
\end{enumerate}
\end{theorem}

Theorem \ref{thm1.1} shows a complete classification on all possible long time behavior of system \eqref{eq1.1} with the assumption ${\bf (J)}$. The criteria governing persistence and extinction of the species show that: (i) When the duration of the bad season is too long (namely, $\rho$ is close to $1$), or the season is too bad (for example, bad weather and food shortages contributes to the large death rate $\delta$) such that $(1-\rho)a-\rho\delta\leq 0$, then the species will die out eventually regardless the initial population size; (ii) If the bad season is not long, or the food resource $a$ is not small such that $\rho\delta<(1-\rho)a\leq (1-\rho)d+\rho\delta$, then both persistence and extinction are determined by the range of the habitat of the species; (iii) When the good season is very long (i.e., $\rho$ is close to $0$), or the species has enough food such that $(1-\rho)(a-d)-\rho\delta>0$, then the species can persist for long time, which is different from that for the local diffusion model \eqref{eq1.5}.

The following conclusion concerns the asymptotic profile of the $\omega$-periodic positive solution $u_{(l_1,l_2)}^*$ of \eqref{eq1.6}.
\begin{theorem}\label{thm1.2}
Assume that ${\bf (J)}$ holds. If $(1-\rho)a-\rho\delta>0$, then there exists $\hat\ell>0$ such that $\lambda_1(-L_{(l_1,l_2)})<0$ for every interval $(l_1,l_2)$ with $l_2-l_1>\hat\ell$ and hence \eqref{eq1.6} admits a unique positive $\omega$-periodic solution $u_{(l_1,l_2)}^*(t,x)$. Moreover,
\[ \lim_{-l_1,l_2\to+\infty}u_{(l_1,l_2)}^*(t,x)=z^*(t)\ \ \mathrm{in}\ C_{\mathrm{loc}}([0,\omega]\times\mathbb R), \]
where $z^*(t)$ is the unique $\omega$-periodic positive solution of the following equation
  \begin{equation}\label{eq1.7}
  \begin{cases}
  z_t=-\delta z, & i\omega< t\leq (i+\rho)\omega, \\
  z_t=z(a-bz), & (i+\rho)\omega<t\leq (i+1)\omega, \\
  z(t+\omega)=z(t), & t\geq 0.
  \end{cases}
  \end{equation}
\end{theorem}

The rest part of this paper is organized as follows. Sections \ref{sec2} are devoted to the global existence and uniqueness of solution of \eqref{eq1.1}. In Section \ref{sec3}, we then study the long-time dynamical behavior of system \eqref{eq1.1} based on the results for the time-periodic eigenvalue problem and time periodic upper-lower solutions. We also show some discussion in the final section.

\section{Well-posedness\label{sec2}}
\setcounter{equation}{0}

In this section, we show the existence and uniqueness of the global solution of \eqref{eq1.1}. Before the statement of well-posedness of solution to \eqref{eq1.1}, we provide a maximum principle.
\begin{lemma}[Maximum principle]\label{lem3.1}
Let $m$ be a positive integer. Assume that ${\bf (J)}$ holds and $-\infty<l_1<l_2<+\infty$. Suppose that $v,v_t\in C([0,m\omega]\times[l_1,l_2]), c\in L^{\infty}([0,m\omega]\times[l_1,l_2])$ and
\begin{equation}\label{eq2e1}
\begin{cases}
v_t\geq -\delta v, & i\omega<t\leq (i+\rho)\omega,\ l_1\leq x\leq l_2, \\
\displaystyle v_t\geq d\int_{l_1}^{l_2}J(x-y)v(t,y)\mathrm dy-dv(t,x)+c(t,x)v, & (i+\rho)\omega<t\leq (i+1)\omega,\ l_1\leq x\leq l_2, \\
v(0,x)\geq 0, & x\in[l_1,l_2],
\end{cases}
\end{equation}
where $i=0,1,\cdots,m-1$. Then $v(t,x)\geq 0$ for $(t,x)\in [0,m\omega]\times[l_1,l_2]$. Moreover, if $v(0,x)\not\equiv 0$ in $[l_1,l_2]$, then $v(t,x)>0$ for $(t,x)\in (\rho\omega,m\omega]\times(l_1,l_2)$; if $v(0,x)>0$ in $(l_1,l_2)$, then $v(t,x)>0$ for $(t,x)\in (0,m\omega]\times(l_1,l_2)$.
\end{lemma}
\begin{proof}
Let $V(t,x)=e^{kt}v(t,x)$. Then $V(0,\cdot)\geq 0$ and $V(t,x)$ satisfies
\begin{equation}\label{eq2e2}
\begin{cases}
V_t\geq p_0V(t,x), & i\omega<t\leq (i+\rho)\omega,\ x\in[l_1,l_2], \\
\displaystyle V_t\geq d\int_{l_1}^{l_2}J(x-y)V(t,y)\mathrm dy+p_1(t,x)V(t,x), & (i+\rho)\omega<t\leq (i+1)\omega,\ x\in[l_1,l_2],
\end{cases}
\end{equation}
where $p_0=k-\delta$, $p_1(t,x)=k+c(t,x)-d$. Due to the boundedness of $c$, there exists $k>0$ such that
\[ p_0>0\ \mathrm{and}\ \inf_{t\in[0,m\omega],x\in[l_1,l_2]}p_1(t,x)>0. \]
We now claim that $V(t,x)\geq 0$ in $[0,m\omega]\times [l_1,l_2]$.

Let $p_{1,0}=\sup_{t\in[0,m\omega],t\in[l_1,l_2]}p_1(t,x)$ and $T_0=\min\Big\{ m\omega,\frac{1}{2(p_0+d+p_{1,0})} \Big\}$. In the following, we will show that the claim holds for $t\in(0,T_0],x\in[l_1,l_2]$. Assume to the contrary that $V_{\mathrm{inf}}:=\inf_{t\in(0,T_0),x\in[l_1,l_2]}V(t,x)<0$. Then there exists $(t_0,x_0)\in(0,T_0]\times[l_1,l_2]$ such that $V_{\mathrm{inf}}=V(t_0,x_0)<0$. Notice that there are $t_n\in (0,t_0]$ and $x_n\in[l_1,l_2]$ such that
\[ V(t_n,x_n)\to V_{\mathrm{inf}}\ \ \mathrm{as}\ \ n\to\infty. \]
We only need to consider the following two cases.

{\it Case 1}. $t_0\in (i_0\omega,(i_0+\rho)\omega]$ for some $i_0\in\{0,1,\cdots,m-1\}$.

In this case, $t_n\in(i_0\omega,t_0]$ for large $n$. Then it follows from \eqref{eq2e2} that
\begin{align*}
V(t_n,x_n)-V(0,x_n)
&=\sum_{i=0}^{i_0-1}\left( \int_{i\omega}^{(i+\rho)\omega}V_t\mathrm dt+\int_{(i+\rho)\omega}^{(i+1)\omega}V_t\mathrm dt \right)+\int_{i_0\omega}^{t_n}V_t\mathrm dt \\
&\geq \sum_{i=0}^{i_0-1}\int_{i\omega}^{(i+\rho)\omega}p_0V(t,x_n)\mathrm dt+\int_{i_0\omega}^{t_n}p_0V(t,x_n)\mathrm dt  \\
&\phantom{=\ }+\sum_{i=0}^{i_0-1}\int_{(i+\rho)\omega}^{(i+1)\omega}\Big[ d\int_{l_1}^{l_2}J(x_n-y)V(t,y)\mathrm dy+p_1(t,x_n)V(t,x_n) \Big]\mathrm dt  \\
&\geq \int_0^{t_n}p_0V_{\mathrm{inf}}\mathrm dt+d\int_0^{t_n}\int_{l_1}^{l_2}J(x_n-y)V_{\mathrm{inf}}\mathrm dy\mathrm dt+\int_0^{t_n}p_{1,0}V_{\mathrm{inf}}\mathrm dt  \\
&\geq t_n(p_0+d+p_{1,0})V_{\mathrm{inf}}  \\
&\geq t_0(p_0+d+p_{1,0})V_{\mathrm{inf}}
\end{align*}
for large $n$. Recall that $V(0,x_n)\geq 0$ for $n=0,1,2,\cdots$. Thus we have
\[ V(t_n,x_n)\geq t_0(p_0+d+p_{1,0})V_{\mathrm{inf}} \]
for large $n$. Taking the limit as $n\to \infty$, it holds that
\[ V_{\mathrm{inf}}\geq t_0(p_0+d+p_{1,0})V_{\mathrm{inf}}\geq \frac{1}{2}V_{\mathrm{inf}}, \]
which is a contradiction.

{\it Case 2}. $t_0\in ((i_0+\rho)\omega,(i_0+1)\omega]$ for some $i_0\in\{0,1,\cdots,m-1\}$.

Similarly, we can also derive a contradiction since
\begin{align*}
V(t_n,x_n)-V(0,x_n)
&=\sum_{i=0}^{i_0-1}\left( \int_{i\omega}^{(i+\rho)\omega}V_t\mathrm dt+\int_{(i+\rho)\omega}^{(i+1)\omega}V_t\mathrm dt \right)+\int_{i_0\omega}^{(i_0+\rho)\omega}V_t\mathrm dt+\int_{(i_0+\rho)\omega}^{t_n}V_t\mathrm dt \\
&\geq t_n(p_0+d+p_{1,0})V_{\mathrm{inf}}\geq t_0(p_0+d+p_{1,0})V_{\mathrm{inf}}
\end{align*}
for large $n$. Therefore, $V(t,x)\geq 0$ for $(t,x)\in (0,T_0]\times[l_1,l_2]$ and then $v(t,\cdot)\geq 0$ for $t\in [0,T_0]$.

If $T_0=m\omega$, then $v(t,x)\geq 0$ in $[0,m\omega]\times[l_1,l_2]$ follows directly; while if $T_0<m\omega$, we can repeat the above process by replacing $V(0,\cdot)$ and $(0,T_0]$ as $V(T_0,\cdot)$ and $(T_0,m\omega]$. Obviously, this process can be repeated in finite many times, and consequently, $v(t,\cdot)\geq 0$ for $t\in [0,m\omega]$.

Now we assume that $v(0,x)\not\equiv 0$ in $[l_1,l_2]$. To finish the proof, it suffices to prove that $V>0$ in $(\rho\omega,\omega]\times(l_1,l_2)$. Suppose that there exists a point $(t_*,x_*)\in(\rho\omega,\omega]\times(l_1,l_2)$ such that $V(t_*,x_*)=0$.

First, we prove that
\[ V(t_*,x)=0\ \ \mathrm{for}\ \ x\in(l_1,l_2). \]
Otherwise, we can find
\[ \tilde x\in[l_1,l_2]\cap\partial\{ x\in(l_1,l_2):V(t_*,x)>0 \}. \]
Then $V(t_*,\tilde x)=0$ and it follows from \eqref{eq2e2} that
\[ 0\geq V_t(t_*,\tilde x)\geq d\int_{l_1}^{l_2}J(\tilde x-y)V(t_*,y)\mathrm dy>0, \]
by assumption ${\bf (J)}$. This is impossible, and hence $V(t_*,x)=0$ for $x\in(l_1,l_2)$. Thus, we can derive from \eqref{eq2e2} that for $x\in[l_1,l_2]$
\begin{align*}
-V(0,x)&=V(t_*,x)-V(0,x)=\int_0^{\rho\omega}V_t\mathrm dt+\int_{\rho\omega}^{t_*}V_t\mathrm dt  \\
&\geq p_0\int_0^{\rho\omega}V(t,x)\mathrm dt+d\int_{\rho\omega}^{t_*}\int_{l_1}^{l_2}J(x-y)V(t,y)\mathrm dy\mathrm dt+\int_{\rho\omega}^{t_*}p_1(t,x)V(t,x)\mathrm dt\geq 0.
\end{align*}
This means that $v(0,x)\equiv 0$ in $[l_1,l_2]$, which is a contradiction.
\end{proof}

\begin{lemma}[Existence and uniqueness]\label{lem3.2}
Assume that ${\bf (J)}$ holds and $-\infty<l_1<l_2<+\infty$. Then for any nonnegative and bounded initial value $u_0(x)\in C([l_1,l_2])$, problem \eqref{eq1.1} admits a unique global solution $u\in C^{1,0}((i\omega,(i+\rho)\omega]\times[l_1,l_2])\cap C^{1,0}(((i+\rho)\omega,(i+1)\omega]\times[l_1,l_2])$ for $i\in \mathbb Z_+$. Moreover, $u(t,x)>0$ for $t>0$ and $x\in(l_1,l_2)$, if $u_0(x)>0$ in $(l_1,l_2)$.
\end{lemma}
\begin{proof}
At first, we set
\[ \hat u=e^{-\delta t}u_0(x) \]
for $t\in[0,\rho\omega]$. Then $\hat u\in C^{1,0}((0,\rho\omega]\times[l_1,l_2])$ satisfies
\[\begin{cases}
\hat u_t=-\delta \hat u, & 0<t\leq \rho\omega,\ l_1\leq x\leq l_2, \\
\hat u(0,x)=u_0(x), & l_1\leq x\leq l_2.
\end{cases}\]
Consider the following problem
\begin{equation}\label{eq2e3}
\begin{cases}
\displaystyle u_t=d\int_{l_1}^{l_2}J(x-y)u(t,y)\mathrm dy-du(t,x)+u(a-bu), & \rho\omega<t\leq \omega,\ x\in (l_1,l_2), \\
u(\rho\omega,x)=e^{-\delta \rho\omega}u_0(x), & x\in[l_1,l_2].
\end{cases}
\end{equation}
Then one can apply the Banach's fixed theorem and comparison argument (see \cite{AMRT10}) to conclude that \eqref{eq2e3} has a unique solution $\bar u(t,x)\in C^{1,0}((\rho\omega,\omega]\times[l_1,l_2])$. Moreover, by the Maximum principle and comparison argument, we have that
\[ 0<\bar u(t,x)\leq \max\left\{ \frac{a}{b},\max_{-h_0\leq x\leq h_0}u_0(x) \right\}\ \mathrm{for}\ t\in (\rho\omega,\omega], x\in(l_1,l_2). \]
Define
\[ u(t,x)=\left\{\begin{aligned}
&\hat u(t,x)\ \mathrm{in}\ [0,\rho\omega]\times[l_1,l_2], \\
&\bar u(t,x)\ \mathrm{in}\ [\rho\omega,\omega]\times[l_1,l_2].
\end{aligned}\right. \]
We have that $u\in \hat u\in C^{1,0}((0,\rho\omega]\times[l_1,l_2])\cap C^{1,0}((\rho\omega,\omega]\times[l_1,l_2])$.

Based on the above obtained function $u$, we let
\[ u_1(t,x)=e^{-\delta(t-\omega)}u(\omega,x) \]
for $\omega\leq t\leq (1+\rho)\omega$. Then $u_1\in C^{1,0}((\omega,(1+\rho)\omega]\times[l_1,l_2])$ satisfies
\[\begin{cases}
u_{1,t}=-\delta u_1, & \omega<t\leq (1+\rho)\omega,\ l_1\leq x\leq l_2, \\
u_1(\omega,x)=u(\omega,x), & l_1\leq x\leq l_2.
\end{cases}\]
Likewise, the nonlocal dispersal problem
\[ \begin{cases}
\displaystyle u_t=d\int_{l_1}^{l_2}J(x-y)u(t,y)\mathrm dy-du(t,x)+u(a-bu), & (1+\rho)\omega<t\leq 2\omega,\ x\in (l_1,l_2), \\
u((1+\rho)\omega,x)=e^{-\delta \rho\omega}u(\omega,x), & x\in[l_1,l_2]
\end{cases} \]
has a unique solution $\bar u_1\in C^{1,0}(((1+\rho)\omega,2\omega]\times[l_1,l_2])$, in which
\[ 0<\bar u_1(t,x)\leq \max\left\{ \frac{a}{b},\max_{-h_0\leq x\leq h_0}u_0(x) \right\}\ \mathrm{for}\ t\in ((1+\rho)\omega,2\omega], x\in(l_1,l_2). \]
Define
\[ u(t,x)=\left\{\begin{aligned}
&u(t,x)\ \mathrm{in}\ [0,\omega]\times[l_1,l_2], \\
&u_1(t,x)\ \mathrm{in}\ [\omega,(1+\rho)\omega]\times[l_1,l_2], \\
&\bar u_1(t,x)\ \mathrm{in}\ [(1+\rho)\omega,2\omega]\times[l_1,l_2].
\end{aligned}\right. \]
Then it holds that $u\in C^{1,0}((i\omega,(i+\rho)\omega]\times[l_1,l_2])\cap C^{1,0}(((i+\rho)\omega,(i+1)\omega]\times[l_1,l_2])$ for $i=0,1$.

By repeating the above procedure, we therefore obtain the existence and uniqueness of the solution $(u,g,h)$ of \eqref{eq1.1}.
\end{proof}

\section{Global dynamics\label{sec3}}
\setcounter{equation}{0}

In this subsection, we first establish the periodic upper-lower solutions method for model \eqref{eq1.1}. Using this method, we can consider the long time behavior of model \eqref{eq1.1}.

\subsection{The method of periodic upper-lower solutions}

Following Hess \cite{He91}, we can define the upper-lower solutions of \eqref{eq1.6} as follows.
\begin{definition}\label{def3.1}
A bounded and continuous function $\tilde u(t,x)$ is called an upper-solution of \eqref{eq1.6} if $\tilde u(t,x)\in C^{1,0}((i\omega,(i+\rho)\omega]\times[l_1,l_2])\cap C^{1,0}(((i+\rho)\omega,(i+1)\omega]\times[l_1,l_2])$ satisfies
\begin{equation}\label{eq3.1}
\begin{cases}
\tilde u_t\geq -\delta \tilde u, & i\omega<t\leq (i+\rho)\omega,\ l_1\leq x\leq l_2, \\
\displaystyle \tilde u_t\geq d\int_{l_1}^{l_2}J(x-y)\tilde u(t,y)\mathrm dy-d\tilde u(t,x)+\tilde u(a-b\tilde u), & (i+\rho)\omega<t\leq (i+1)\omega,\ l_1\leq x\leq l_2, \\
\tilde u(0,x)\geq \tilde u(\omega,x), & x\in[l_1,l_2].
\end{cases}
\end{equation}
for $i\in \mathbb Z_+$. Meanwhile, the function $\hat u(t,x)\in C^{1,0}((i\omega,(i+\rho)\omega]\times[l_1,l_2])\cap C^{1,0}(((i+\rho)\omega,(i+1)\omega]\times[l_1,l_2])$ is called a lower-solution of \eqref{eq1.6} if the inequalities in \eqref{eq3.1} are reversed.
\end{definition}

Similarly, we can define the upper-solution (resp. lower-solution) of \eqref{eq1.1} by replacing the inequality $\tilde u(0,x)\geq \tilde u(\omega,x)$ in \eqref{eq3.1} as $\tilde u(0,x)\geq \tilde u_0(x)$ (resp. $\hat u(0,x)\leq \hat u_0(x)$). We say that a pair of upper-lower solution $\tilde u$ and $\hat u$ are ordered if $\tilde u(t,x)\geq \hat u(t,x)$ in $[0,+\infty)\times [l_1,l_2]$.

Using the semigroup theory, we have the following result.
\begin{lemma}\label{lem3.4}
Let $D(t),\bar a(t)$ and $\bar b(t)$ be defined as in \eqref{eq1.2}. Assume that $u(t,x)$ is bounded for $(t,x)\in[0,+\infty]\times [l_1,l_2]$. Then $u(t,x)$ is a solution of \eqref{eq1.1} if and only if
\begin{equation}\label{eq3.2}
\begin{aligned}
\displaystyle u(t,x)&=u(0,x)+\int_0^t\bigg[D(s)\bigg( \int_{l_1}^{l_2}J(x-y)u(s,y)\mathrm dy-u(s,x) \bigg) \\
\displaystyle &\phantom{=}+u(s,x)[\bar a(s)-\bar b(s)u(s,x)] \bigg] \mathrm ds, & t>0,\ x\in[l_1,l_2].
\end{aligned}
\end{equation}
\end{lemma}
\begin{proof}
It follows from the semigroup method \cite{P83}, we have that
\begin{equation}\label{eq3.3}
\begin{aligned}
&u(t,x)=e^{-t}u(0,x) \\
&\displaystyle \phantom{==}+\int_0^te^{-(t-s)}\bigg[D(s)\bigg( \int_{l_1}^{l_2}J(x-y)u(s,y)\mathrm dy-u(s,x) \bigg)+u(s,x)+u(s,x)[\bar a(s)-\bar b(s)u(s,x)] \bigg] \mathrm ds,  \\
\end{aligned}
\end{equation}
which implies
\begin{equation}\label{eq3.4}
\int_0^tu(s,x)\mathrm ds=(1-e^{-t})u(0,x)+I_1[u](t,x),
\end{equation}
where
\begin{align*}
I_1[u](t,x)&=\int_0^t\int_0^se^{-(s-z)}\bigg[D(z)\bigg( \int_{l_1}^{l_2}J(x-y)u(z,y)\mathrm dy-u(z,x) \bigg)+u(z,x)[1+\bar a(z)-\bar b(z)u(z,x)] \bigg] \mathrm dz\mathrm ds  \\
&=\int_0^t\int_z^te^{-(s-z)}\bigg[D(z)\bigg( \int_{l_1}^{l_2}J(x-y)u(z,y)\mathrm dy-u(z,x) \bigg)+u(z,x)[1+\bar a(z)-\bar b(z)u(z,x)] \bigg] \mathrm ds\mathrm dz  \\
&\phantom{=\ }+\int_0^t(1-e^{s-t})\bigg[D(s)\bigg( \int_{l_1}^{l_2}J(x-y)u(s,y)\mathrm dy-u(s,x) \bigg)+u(s,x)[1+\bar a(s)-\bar b(s)u(s,x)] \bigg] \mathrm ds.
\end{align*}
Therefore, \eqref{eq3.2} can be derived from \eqref{eq3.3} and \eqref{eq3.4}. On the other hand, if $u$ satisfies \eqref{eq3.2}, then we can also show that \eqref{eq3.3} holds.
\end{proof}

Similarly, we have the following result for \eqref{eq1.6}.
\begin{lemma}\label{lem3.5}
Assume that $u(t,x)$ is bounded for $(t,x)\in[0,+\infty]\times [l_1,l_2]$. Then $u(t,x)$ is a solution of \eqref{eq1.6} if and only if
\begin{equation}\label{eq3.5}
\begin{cases}
\displaystyle u(t,x)=u(0,x)+\int_0^t\bigg[D(s)\bigg( \int_{l_1}^{l_2}J(x-y)u(s,y)\mathrm dy-u(s,x) \bigg) & \\
\displaystyle \phantom{=====}+u(s,x)[\bar a(s)-\bar b(s)u(s,x)] \bigg] \mathrm ds, & t>0,\ x\in[l_1,l_2], \\
u(t,x)=u(t+\omega,x), & t\geq 0,\ x\in[l_1,l_2].
\end{cases}
\end{equation}
\end{lemma}

\begin{corollary}\label{cor3.6}
Assume that $u(t,x)$ is bounded for $(t,x)\in[0,+\infty]\times [l_1,l_2]$. Then $u(t,x)$ is a solution of \eqref{eq1.6} if and only if
\begin{equation}\label{eq3.6}
\begin{cases}
\displaystyle u(t,x)=e^{-Ct}u(0,x)+\int_0^te^{-C(t-s)}\bigg[D(s)\bigg( \int_{l_1}^{l_2}J(x-y)u(s,y)\mathrm dy-u(s,x) \bigg) & \\
\displaystyle \phantom{=====}+u(s,x)[C+\bar a(s)-\bar b(s)u(s,x)] \bigg] \mathrm ds, & t>0,\ x\in[l_1,l_2], \\
u(t,x)=u(t+\omega,x), & t\geq 0,\ x\in[l_1,l_2],
\end{cases}
\end{equation}
where $C$ is a constant.
\end{corollary}

The following conclusion establishes a method of periodic upper-lower solutions.
\begin{theorem}\label{thm3.7}
Assume that ${\bf (J)}$ holds and $u_0(x)\in C([l_1,l_2])$ is bounded. Let $u(t,x)$ be the unique solution to \eqref{eq1.1} and $\tilde u(t,x),\hat u(t,x)$ be a pair of ordered and bounded upper-lower solutions to \eqref{eq1.6} satisfying
\[ \hat u(0,x)\leq u_0(x)\leq \tilde u(0,x)\ \ \mathrm{on}\ [l_1,l_2]. \]
Then the time periodic problem \eqref{eq1.6} admits a minimal solution $\underline u\in C^{1,0}((i\omega,(i+\rho)\omega]\times[l_1,l_2])\cap C^{1,0}(((i+\rho)\omega,(i+1)\omega]\times[l_1,l_2])$ and a maximal solution $\overline u\in C^{1,0}((i\omega,(i+\rho)\omega]\times[l_1,l_2])\cap C^{1,0}(((i+\rho)\omega,(i+1)\omega]\times[l_1,l_2])$ $(i\in Z_+)$ satisfying
\[ \hat u(t,x)\leq \underline u(t,x)\leq \varliminf_{n\to\infty}u(t+n\omega,x)\leq \varlimsup_{n\to\infty}u(t+n\omega,x)\leq \overline u(t,x)\leq \tilde u(t,x)\ \ \mathrm{for}\ t\geq 0,x\in[l_1,l_2]. \]
\end{theorem}
\begin{proof}
For notational convenience, denote $Q=(0,+\infty)\times[l_1,l_2]$ and $J*u(t,x)=\int_{l_1}^{l_2}J(x-y)u(t,y)\mathrm dy$ for $u\in C(\overline Q)$. Set
\[ I=\Big[ -\|\hat u\|_{L^{\infty}(Q)}-\|\tilde u\|_{L^{\infty}(Q)},\|\hat u\|_{L^{\infty}(Q)}+\|\tilde u\|_{L^{\infty}(Q)} \Big]. \]
At first we take $K>1$ such that for $u\in I$, $u[\bar a(t)-\bar bu]+Ku$ and $\bar a(t)u-2(\|\hat u\|_{L^{\infty}(Q)}+\|\tilde u\|_{L^{\infty}(Q)})\bar b(t) u+Ku$ are both increasing with respect to $u$. Define
\[ \mathcal L[u](t,x)=u(t,x)-e^{-Kt}u(0,x)-\int_0^te^{-K(t-s)}D(s)\left[ J*u(s,x)-u(s,x) \right]\mathrm ds \]
and
\[ F(t,u)=u(t,x)[\bar a(t)-\bar b(t)u(t,x)]+Ku(t,x). \]
We construct two iterations sequences by the following linear nonlocal evolution equations
\begin{equation}\label{eq3.7}
\begin{cases}
\mathcal L[\overline u^n](t,x)=\int_0^te^{-K(t-s)}F(s,\overline u^{n-1}(s,x))\mathrm ds, & (t,x)\in Q, \\
\overline u^n(0,x)=\overline u^{n-1}(\omega,x), & x\in [l_1,l_2]
\end{cases}
\end{equation}
and
\begin{equation}\label{eq3.8}
\begin{cases}
\mathcal L[\underline u^n](t,x)=\int_0^te^{-K(t-s)}F(s,\underline u^{n-1}(s,x))\mathrm ds, & (t,x)\in Q, \\
\underline u^n(0,x)=\underline u^{n-1}(\omega,x), & x\in [l_1,l_2],
\end{cases}
\end{equation}
where $n\geq 1,\overline u^0(t,x)=\tilde u(t,x)$ and $\underline u^0(t,x)=\hat u(t,x)$. We can check that a sufficiently large constant is an upper-solution of \eqref{eq3.7} (resp. \eqref{eq3.8}) since $K>1$. Then an application of Banach's fixed point theorem and comparison principle yields that the linear initial value problem \eqref{eq3.7} (resp. \eqref{eq3.8}) has a unique bounded global solution $\overline u^n(t,x)$ (resp. $\underline u^n(t,x)$) for any $n\geq 1$. We complete the proof of this theorem by the following four steps.

{\it Step 1.} The sequences $\{\overline u^n\}_{n=1}^{\infty}$ and $\{\underline u^n\}_{n=1}^{\infty}$ satisfy
\begin{equation}\label{eq3.9}
\hat u(t,x)\leq \underline u^n(t,x)\leq \underline u^{n+1}(t,x)\leq u(t+n\omega,x)\leq \overline u^{n+1}(t,x)\leq \overline u^n(t,x)\leq \tilde u(t,x),\ \ (t,x)\in \overline Q,
\end{equation}
for $n\geq 1$.

Since $\tilde u(t,x)$ is a bounded upper-solution of \eqref{eq1.6} and $\tilde u(0,x)\geq u_0(x)$, we see that $\tilde u(t,x)$ is also a bounded upper-solution of \eqref{eq1.1}. Then by Lemma \ref{lem3.1}, we have $\tilde u(t,x)\geq u(t,x)$ in $\overline Q$. It follows from \eqref{eq3.7} and Corollary \ref{cor3.6} that $\overline u^1(t,x)$ satisfies
\begin{equation}\label{eq3.10}
\begin{cases}
\overline u_t^1(t,x)=D(t)\left[ J*\overline u^1(t,x)-\overline u^1(t,x) \right]+\tilde u(t,x)[\bar a(t)-\bar b(t)\tilde u(t,x)]+K[\tilde u(t,x)-\overline u^1(t,x)], & (t,x)\in Q, \\
\overline u^1(0,x)=\tilde u(\omega,x), & x\in [l_1,l_2].
\end{cases}
\end{equation}
Set $w^1(t,x)=\overline u^0(t,x)-\overline u^1(t,x)=\tilde u(t,x)-\overline u^1(t,x)$. Since
\[ \begin{cases}
\tilde u_t(t,x)\geq D(t)\left[ J*\tilde u(t,x)-\tilde u(t,x) \right]+\tilde u(t,x)[\bar a(t)-\bar b(t)\tilde u(t,x)], & (t,x)\in Q, \\
\tilde u(0,x)\geq \tilde u(\omega,x), & x\in [l_1,l_2],
\end{cases} \]
there holds that
\[ \begin{cases}
w^1_t(t,x)\geq D(t)\left[ J*w^1(t,x)-w^1(t,x) \right]-Kw^1(t,x), & (t,x)\in Q, \\
w^1(0,x)\geq 0, & x\in [l_1,l_2],
\end{cases} \]
which together with Lemma \ref{lem3.1} implies that $w^1(t,x)\geq 0$ and so
\[ \overline u^1(t,x)\leq \overline u^0(t,x)=\tilde u(t,x)\ \ \mathrm{for}\ (t,x)\in\overline Q. \]
By a similar manner for lower-solution $\hat u(t,x)$, we have
\[ u(t,x)\geq \hat u(t,x)\ \ \mathrm{and}\ \ \underline u^1(t,x)\geq \underline u^0(t,x)=\hat u(t,x)\ \ \mathrm{for}\ (t,x)\in\overline Q. \]

Now we let $w^2(t,x)=\overline u^1(t,x)-\underline u^1(t,x)$. In view of $\hat u(t,x)\leq u(t,x)\leq \tilde u(t,x)$ in $\overline Q$, by \eqref{eq3.7} and \eqref{eq3.8}, we have $w^2(0,x)=\tilde u(\omega,x)-\hat u(\omega,x)\geq 0$ and
\begin{align*}
w^2_t(t,x)&=D(t)\left[ J*w^2(t,x)-w^2(t,x) \right]-Kw^2(t,x)  \\
&\phantom{=\ }+\bar a(t)[\tilde u(t,x)-\hat u(t,x)]-\bar b(t)[\tilde u^2(t,x)-\hat u^2(t,x)]+K[\tilde u(t,x)-\hat u(t,x)]  \\
&\geq D(t)\left[ J*w^2(t,x)-w^2(t,x) \right]-Kw^2(t,x)\ \ \mathrm{in}\ \overline Q,
\end{align*}
where the conditions satisfied by $K$ are used here. It follows from Lemma \ref{lem3.1} that $w^2(t,x)\geq 0$ and hence $\overline u^1(t,x)\geq \underline u^1(t,x)$ in $\overline Q$.

Next, we show that $\underline u^1(t,x)\leq u(t+\omega)\leq \overline u^1(t,x)$ in $\overline Q$. Let $w^3(t,x)=\overline u^1(t,x)-u(t+\omega,x)$. Notice that $u(t+\omega,x)$ satisfies
\begin{align}
u_t(t+\omega,x)
&= D(t+\omega)\left[ J*\tilde u(t+\omega,x)-\tilde u(t+\omega,x) \right]+\tilde u(t+\omega,x)[\bar a(t+\omega)-\bar b(t+\omega)\tilde u(t+\omega,x)] \notag \\
&=D(t)\left[ J*\tilde u(t+\omega,x)-\tilde u(t+\omega,x) \right]+\tilde u(t+\omega,x)[\bar a(t)-\bar b(t)\tilde u(t+\omega,x)]\ \ \mathrm{in}\ \overline Q. \label{eq3.11}
\end{align}
Combining \eqref{eq3.10} and \eqref{eq3.11}, there holds that $w^3(0,x)=\overline u^1(0,x)-u(\omega,x)=\tilde u(\omega,x)-u(\omega,x)\geq 0$ for $x\in[l_1,l_2]$ and
\begin{align*}
w^3_t(t,x)&=D(t)\left[ J*w^3(t,x)-w^3(t,x) \right]+\tilde u(t,x)\big[\bar a(t)-\bar b(t)\tilde u(t,x)\big]+K\big[\tilde u(t,x)-\overline u^1(t,x)\big]  \\
&\phantom{=\ }-u(t+\omega)\big[\bar a(t)-\bar b(t)u(t+\omega,x)\big]  \\
&= D(t)\left[ J*w^3(t,x)-w^3(t,x) \right]+\big[\bar a(t)-\bar b(t)\big(\overline u^1(t,x)+u(t+\omega,x)\big)\big]w^3(t,x)  \\
&\phantom{=\ }+\bar a(t)\big[\tilde u(t,x)-\overline u^1(t,x)\big]-\bar b(t)\big[\tilde u(t,x)+\overline u^1(t,x)\big]\big[\tilde u(t,x)-\overline u^1(t,x)\big]+K\big[\tilde u(t,x)-\overline u^1(t,x)\big].
\end{align*}
Since $\tilde u(t,x)\geq \overline u^1(t,x)\geq \underline u^1(t,x)\geq \hat u(t,x)$, by the condition satisfied by $K$, we see that
\[ \bar a(t)\big[\tilde u(t,x)-\overline u^1(t,x)\big]-\bar b(t)\big[\tilde u(t,x)+\overline u^1(t,x)\big]\big[\tilde u(t,x)-\overline u^1(t,x)\big]+K\big[\tilde u(t,x)-\overline u^1(t,x)\big]\geq 0\ \ \mathrm{in}\ \overline Q, \]
which leads to that
\[ w^3_t(t,x)\geq D(t)\left[ J*w^3(t,x)-w^3(t,x) \right]+\big[\bar a(t)-\bar b(t)\big(\overline u^1(t,x)+u(t+\omega,x)\big)\big]w^3(t,x)\ \ \mathrm{in}\ \overline Q. \]
Due to the boundedness of $u^1(t,x)$ and $u(t+\omega,x)$, we can derive from Lemma \ref{lem3.1} that $w^3(t,x)\geq 0$ and then $\overline u^1(t,x)\geq u(t+\omega,x)$ in $\overline Q$. Similarly, we also have $\underline u^1(t,x)\leq u(t+\omega,x)$ in $\overline Q$. Therefore, the following inequalities are true:
\[ \hat u(t,x)\leq \underline u^1(t,x)\leq u(t+\omega,x)\leq \overline u^1(t,x)\leq \tilde u(t,x),\ \ (t,x)\in \overline Q. \]

An induction argument implies the monotone property \eqref{eq3.9} immediately. Since $\{\overline u^n\}$ and $\{\underline u^n\}$ monotonically bounded sequences, there exist two bounded function $\overline u(t,x)$ and $\underline u(t,x)$ such that
\[ \lim_{n\to\infty}\overline u^n(t,x)=\overline u(t,x)\ \mathrm{and}\ \lim_{n\to\infty}\underline u^n(t,x)=\underline u(t,x) \]
and
\[ \underline u(t,x)\leq \varliminf_{n\to\infty}u(t+n\omega,x)\leq \varlimsup_{n\to\infty}u(t+n\omega,x)\leq \overline u(t,x) \]
for each $(t,x)\in\overline Q$. Thus from the dominated convergence theorem, we obtain that $\overline u(t,x)$ and $\underline u(t,x)$ are bounded solutions of the initial value problem
\[ \begin{cases}
u_t(t,x)=D(t)\left[ J*u(t,x)-u(t,x) \right]+u(t,x)[\bar a(t)-\bar b(t)u(t,x)], & (t,x)\in Q, \\
u(0,x)=u(\omega,x), & x\in [l_1,l_2].
\end{cases} \]

{\it Step 2.} We prove that $\overline u(t,x),\underline u(t,x)\in C^{1,0}((i\omega,(i+\rho)\omega]\times[l_1,l_2])\cap C^{1,0}(((i+\rho)\omega,(i+1)\omega]\times[l_1,l_2])$ for all $i\in \mathbb Z_+$.

Since $\overline u(t,x)=0$ in $Q'$ and
\[ \overline u(t,x)=\overline u(0,x)+\int_0^t\Big[D(s)\big[J*\overline u(s,x)-\overline u(s,x)\big]+\overline u(s,x)\big[ \bar a(s)-\bar b(s)\overline u(s,x) \big] \Big]\mathrm ds, \]
it holds that
\[ \overline u(t+\varepsilon,x)-\overline u(t,x)=\int_t^{t+\varepsilon}\Big[D(s)\big[J*\overline u(s,x)-\overline u(s,x)\big]+\overline u(s,x)\big[ \bar a(s)-\bar b(s)\overline u(s,x) \big] \Big]\mathrm ds \]
for each fixed $(t,x)\in\overline Q$, where $|\varepsilon|>0$ is sufficiently small. Then, we have
\[ |\overline u(t+\varepsilon,x)-\overline u(t,x)|\leq \int_t^{t+\varepsilon}\left|\Big[D(s)\big[J*\overline u(s,x)-\overline u(s,x)\big]+\overline u(s,x)\big[ \bar a(s)-\bar b(s)\overline u(s,x) \big] \Big]\right|\mathrm ds\leq C|\varepsilon|, \]
where $C>0$ is a constant independent of $\varepsilon$. This means that $\overline u(t,x)$ is continuous in $t\in [0,+\infty)$. The continuity of $\overline u(t,x)$ in $x\in[l_1,l_2]$ follows from the argument in \cite{AMRT10}.

For any $t_0\in(0,+\infty)$, there must exist a unique $i_0\in\mathbb Z_+$ such that either $t_0\in(i_0\omega,(i_0+\rho)\omega]$, or $t_0\in((i_0+\rho)\omega,(i_0+1)\omega]$. When $t_0,t_0+\varepsilon\in(i_0\omega,(i_0+\rho)\omega]$, we see that
\begin{align*}
\lim_{\varepsilon\to 0}\dfrac{\overline u(t+\varepsilon,x)-\overline u(t,x)}{\varepsilon}
&=\lim_{\varepsilon\to 0}\frac{1}{\varepsilon}\int_t^{t+\varepsilon}\Big[-\delta\overline u(s,x)\Big]\mathrm ds \\
&=-\delta \lim_{\varepsilon\to 0}\overline u(t+\theta\varepsilon,x)\ (0<\theta<1) \\
&=-\delta \overline u(t,x).
\end{align*}
When $t_0,t_0+\varepsilon \in((i_0+\rho)\omega,(i_0+1)\omega]$, we have
\begin{align*}
\lim_{\varepsilon\to 0}\dfrac{\overline u(t+\varepsilon,x)-\overline u(t,x)}{\varepsilon}
&=\lim_{\varepsilon\to 0}\frac{1}{\varepsilon}\int_t^{t+\varepsilon}\Big[d\big[J*\overline u(s,x)-\overline u(s,x)\big]+\overline u(s,x)\big[ a-b\overline u(s,x) \big] \Big]\mathrm ds \\
&=\lim_{\varepsilon\to 0}d\big[J*\overline u(t+\theta\varepsilon,x)-\overline u(t+\theta\varepsilon,x)\big]+\overline u(t+\theta\varepsilon,x)\big[ a-b\overline u(t+\theta\varepsilon,x) \big]\ (0<\theta<1) \\
&=d\big[J*\overline u(t,x)-\overline u(t,x)\big]+\overline u(t,x)\big[ a-b\overline u(t,x) \big].
\end{align*}
Hence, $\overline u(t,x)\in C^{1,0}((i\omega,(i+\rho)\omega]\times[l_1,l_2])\cap C^{1,0}(((i+\rho)\omega,(i+1)\omega]\times[l_1,l_2])$ for all $i\geq 0$ due to the arbitrariness of $t_0$.

The proof for $\underline u(t,x)$ is similar.

{\it Step 3.} We prove that $\overline u(t,x)=\overline u(t+\omega,x)$ and $\underline u(t,x)=\underline u(t+\omega,x)$ for all $t\geq 0$.

Let $v(t,x)=\overline u(t+\omega,x)-\overline u(t,x)$. Note that $D(t),\bar a(t)$ and $\bar b(t)$ are all $\omega-$periodic in $t$. Then
\begin{equation}\label{eq3.12}
v_t(t,x)=D(t)\left[\int_{\Omega}J(x-y)v(t,y)\mathrm dy-v(t,x) \right]+\bar a(t)v(t,x)-\bar b(t)[\overline u(t+\omega,x)+\overline u(t,x)]v(t,x).
\end{equation}
Since $v(0,x)=\overline u(\omega,x)-\overline u(0,x)=0$, the uniqueness of solution of initial value problem \eqref{eq3.12} implies that $v(t,x)\equiv 0$ in $Q$, equivalently, $\overline u(t,x)$ is $\omega$-periodic in $t$.

Similarly, we can also prove that $\underline u(t,x)=\underline u(t+\omega,x)$, and omit the details here.

{\it Step 4.} We show the maximality of $\overline u(t,x)$ and minimality of $\underline u(t,x)$. Notice that every $\omega$-periodic solution $u^*(t,x)$ of \eqref{eq1.6} satisfies $\hat u(t,x)\leq u^*(t,x)\leq \tilde u(t,x)$. Meanwhile, $u^*(t,x)$ is a lower-solution as well as a upper-solution of \eqref{eq1.6}. By choosing $\tilde u$ and $u^*$ as a pair of upper-lower solutions to \eqref{eq1.6}, there holds that $u^*(t,x)\leq \overline u^n(t,x)\leq \tilde u(t,x)$ and hence $u^*(t,x)\leq \overline u(t,x)\leq \tilde u(t,x)$. On the other hand, if we take $u^*$ and $\hat u$ as a pair of upper-lower solutions to \eqref{eq1.6}, then $\hat u(t,x)\leq \underline u(t,x)\leq u^*(t,x)$.
\end{proof}

\subsection{Proof of Theorems \ref{thm1.1} and \ref{thm1.2}}

In this subsection, we complete the proof of Theorems \ref{thm1.1} and \ref{thm1.2}. Linearizing model \eqref{eq1.1} at zero, we obtain the time-periodic eigenvalue problem
\begin{equation}\label{eq3.13}
\begin{cases}
v_t+\delta v=\lambda v, & i\omega<t\leq (i+\rho)\omega,\ l_1\leq x\leq l_2, \\
\displaystyle v_t-d(J*v-v)(t,x)+av=\lambda v, & (i+\rho)\omega<t\leq (i+1)\omega,\ l_1\leq x\leq l_2.
\end{cases}
\end{equation}
It is well known (see, e.g., \cite{BCV16,C10,CDM13}) that the time independent eigenvalue equation
\begin{equation}\label{eq3.14}
d(J*\phi-\phi)(t,x)+a\phi(x)=-\sigma\phi(x),\quad x\in[l_1,l_2].
\end{equation}
admits a principal eigenvalue $\sigma_1$, which satisfies $\sigma_1<d-a$. Moreover, from \cite[Proposition 3.4]{CDLL19}, we see that
\begin{proposition}\label{prop3.6}
Assume that ${\bf (J)}$ holds and $-\infty<l_1<l_2<+\infty$. Then the following hold true:
\begin{enumerate}[(1)]
\item $\sigma_1$ is strictly decreasing and continuous in $\ell:=l_2-l_1$;
\item $\lim_{l_2-\l_1\to +\infty}\sigma_1=-a$;
\item $\lim_{l_2-\l_1\to 0^+}\sigma_1=d-a$.
\end{enumerate}
\end{proposition}

Let $\phi_1(x)$ be the positive eigenfunction of \eqref{eq3.14} associated with $\sigma_1$. By defining
\[ \sigma(t)=\begin{cases}
\delta, & t\in (i\omega,(i+\rho)\omega], \\
\sigma_1, & t\in ((i+\rho)\omega,(i+1)\omega],
\end{cases} \]
we see that
\[ D(t)\left[ J*\phi_1(x)-\phi_1(x) \right]+\bar a(t)\phi_1(x)=-\sigma(t)\phi_1(x),\ \ \forall t\in\mathbb R,x\in\overline\Omega. \]
Set
\begin{equation}\label{eq3.15}
\varphi(t,x)=\exp\left[ \big((1-\rho)\sigma_1+\rho\delta \big)t-\int_0^t\sigma(s)\mathrm ds \right]\phi_1(x).
\end{equation}
Then there holds that
\[ \begin{cases}
\varphi_t+\delta \varphi=[(1-\rho)\sigma_1+\rho\delta] \varphi, & i\omega<t\leq (i+\rho)\omega,\ l_1\leq x\leq l_2, \\
\displaystyle \varphi_t-d(J*\varphi-\varphi)(t,x)+a\varphi=[(1-\rho)\sigma_1+\rho\delta] \varphi, & (i+\rho)\omega<t\leq (i+1)\omega,\ l_1\leq x\leq l_2, \\
\varphi(t,x)=\varphi(t+\omega,x), & t\geq 0,l_1\leq x\leq l_2.
\end{cases} \]
This means that $(1-\rho)\sigma_1+\rho\delta$ is a eigenvalue of \eqref{eq3.13} with the positive eigenfunction $\varphi(t,x)$. In the proof of Theorems \ref{thm1.1}, we will show that $(1-\rho)\sigma_1+\rho\delta$ serves as a threshold which determines whether the species can persist.
\begin{proof}[Proof of Theorem \ref{thm1.1}]
Let $\lambda_1=(1-\rho)\sigma_1+\rho\delta$. we first consider two cases on the sign of $\lambda_1$.

{\bf Case 1.} Suppose that $\lambda_1<0$.

One can easily check that a sufficiently large positive constant $M$ is a upper-solution of \eqref{eq1.1} as well as the upper-solution of \eqref{eq1.6}. Following the comparison argument in Theorem \ref{thm3.7}, we see that for any $(t,x)\in[0,\omega]\times [l_1,l_2]$, $u(t+n\omega,x;M)$ is non-increasing with respect to $n$. Then the function
\[ u^+(t,x):=\lim_{n\to\infty}u(t+n\omega,x;M),\ \ (t,x)\in[0,\omega]\times [l_1,l_2] \]
is well-defined and upper semi-continuous. On the other hand, let $\varphi(t,x)$ be defined as in \eqref{eq3.15}. For any $0<\varepsilon\ll 1$, by $\lambda_1<0$, we see that $\varepsilon \varphi$ is a lower-solution of \eqref{eq1.1} as well as a lower-solution of \eqref{eq1.6}. Again, by the comparison argument, $u(t+n\omega,x;\varepsilon\varphi(0,x))$ is non-decreasing as $n$ increases for any $(t,x)\in[0,\omega]\times [l_1,l_2]$. Thus, the function
\[ u^-(t,x):=\lim_{n\to\infty}u(t+n\omega,x;\varepsilon\varphi),\ \ (t,x)\in[0,\omega]\times [l_1,l_2] \]
is well-defined and lower semi-continuous. Obviously, $u^-\leq u^+$.

Next, we show $u^-(t,x)\equiv u^+(t,x)$. For this purpose, we define
\[ \gamma_n:=\inf\left\{ \ln\alpha: \frac{1}{\alpha}u(t+n\omega,x;M)\leq u(t+n\omega,x;\varepsilon\varphi(0,x))\leq \alpha u(t+n\omega,x;M),\ (t,x)\in[0,\omega]\times [l_1,l_2] \right\}. \]
Since the sequence $\{u(\cdot+n\omega,\cdot;\varepsilon\varphi(0,\cdot))\}_n$ is non-decreasing and $\{u(\cdot+n\omega,\cdot;M)\}_n$ is non-increasing, $u(\cdot+n\omega,\cdot;\varepsilon\varphi(0,\cdot))$ and $u(\cdot+n\omega,\cdot;M)$ will be closer to each other when $n$ decreases. Consequently, $\{ \gamma_n \}_n$ is a non-increasing sequence, and then the limit $\gamma_*:=\lim_{n\to\infty}\gamma_n$ exists. If $\gamma_*>0$, then by the comparison argument, we can construct some $\alpha_*>1$ and $0<\sigma\ll 1$ such that $\frac{1}{\alpha_*}u(\cdot+n\omega,\cdot;M)\leq u(\cdot+n\omega,\cdot;\varepsilon\varphi(0,\cdot))\leq \alpha_* u(\cdot+n\omega,\cdot;M)$ and $\ln \alpha_*<\gamma_n-\sigma$ for sufficiently large $n$. This causes a contradiction with the definition of $\gamma_*$. Hence, $\gamma_*=0$ and the equality $u^+=u^-$ follows.

Notice that $u^+$ is upper semi-continuous and $u^-$ is lower semi-continuous. Then $u^*:=u^+$ is continuous and $\inf\limits_{[0,\omega]\times[l_1,l_2]}u^*>0$. Using Dini's Theorem, we have $\lim_{n\to\infty}u(\cdot+n\omega,\cdot;M)=\lim_{n\to\infty}u(\cdot+n\omega,\cdot;\varepsilon\varphi(0,\cdot))=u^*$ uniformly for $(t,x)\in [0,\omega]\times[l_1,l_2]$. This also means that
\[ u(t+\omega,x;u^*(0,x))=\lim_{n\to\infty}u(t+\omega,x;u(n\omega,\cdot;M))=\lim_{n\to\infty}u(t+(n+1)\omega,x;M)=u^*(t,x). \]
This is, $u(t,x;u^*(0,x))$ is $\omega$-periodic in $t$. The existence of time periodic positive solution of \eqref{eq1.6} is established.

By the above contraction argument, we can obtain the existence of the solutions of \eqref{eq1.6}. The uniqueness follows directly from Theorem \ref{thm3.7} and the above argument. To emphasize the dependence of $u^*(t,x)$ on $l_1,l_2$, denote by $u_{(l_1,l_2)}^*(t,x)$ the unique time periodic positive solution of \eqref{eq1.6}. Since
\[ \varepsilon\varphi(t,x)\leq u(t,x;u_0)\leq M,\ \ (t,x)\in [0,\omega]\times[l_1,l_2] \]
for $0<\varepsilon\ll 1$ and $M\gg 1$, the above contraction argument also implies that
\[ \lim_{n\to\infty}u(t+n\omega,x;u_0)=u_{(l_1,l_2)}^*(t,x) \]
uniformly in $C([0,\omega]\times [l_1,l_2])$. The global stability of $u_{(l_1,l_2)}^*$ can also be inferred from Theorem \ref{thm3.7} and the uniqueness of the solutions of \eqref{eq1.6}.

{\bf Case 2.} Suppose that $\lambda_1\geq 0$.

At first, we show the nonexistence of positive solution of \eqref{eq1.6}. By way of contradiction, suppose that $v^*$ is a positive solution of \eqref{eq1.6}. Then we can choose $\epsilon>0$ small enough such that $\epsilon\varphi<v^*$ in $[0,+\infty)\times[l_1,l_2]$. There holds that
\begin{align*}
0&\leq \lambda_1\varphi(t,x) \\
&=\partial_t\varphi(t,x)-D(t)\bigg[ \int_{l_1}^{l_2}J(x-y)\varphi(t,y)\mathrm dy-\varphi(t,x) \bigg]-\bar a(t)\varphi(t,x)  \\
&\leq \partial_t\varphi(t,x)-D(t)\bigg[ \int_{l_1}^{l_2}J(x-y)\varphi(t,y)\mathrm dy-\varphi(t,x) \bigg]-\bar a(t)\varphi(t,x)+\bar b(t)\varphi^2(t,x)
\end{align*}
in $[0,+\infty)\times[l_1,l_2]$, which means $\varphi$ is an upper-solution of \eqref{eq1.6}. It follows from the comparison argument that $v^*\leq \varphi$ in $[0,+\infty)\times[l_1,l_2]$. This is a contradiction. Hence the equation \eqref{eq1.6} admits no positive solution.

Since $\lambda_1\geq 0$, a simple calculation gives that for large $e>0$, $e\varphi(t,x)$ and $0$ are a pair of upper-lower solutions of \eqref{eq1.1} as well as a pair of upper-lower solutions of \eqref{eq1.6}. It then follows from Theorem \ref{thm3.7} that the time periodic problem \eqref{eq1.6} admits a minimal solution $\underline u$ and a maximal solution $\overline u$ satisfying
\[ 0\leq \underline u(t,x)\leq \varliminf_{n\to\infty}u(t+n\omega,x;u_0)\leq \varlimsup_{n\to\infty}u(t+n\omega,x;u_0)\leq \overline u(t,x)\leq e\varphi(t,x)\ \ \mathrm{in}\ [0,+\infty)\times[l_1,l_2]. \]
The nonexistence of positive solution to \eqref{eq1.6} implies that $\overline u=\underline u=0$. Thus, the solution $u(t,x;u_0)$ of \eqref{eq1.1} converges to $0$ point by point. Since $\overline u,\underline u\in C([0,+\infty)\times[l_1,l_2])$ and the sequences constructed in \eqref{eq3.9} are monotone, we have $\lim\limits_{t\to\infty}u(t,x;u_0)=0$ uniformly for $x\in[l_1,l_2]$ by Dini's Theorem.

From the above two cases, one can obtain that the sign of $(1-\rho)\sigma_1+\rho\delta$ can completely determine the long time behavior of the species. Therefore, the conclusions of Theorem \ref{thm1.1} follows from the above argument and Proposition \ref{prop3.6}.
\end{proof}

We further discuss the behaviors of the positive $\omega$-periodic solution to \eqref{eq1.6}. Look at the ODE system \eqref{eq1.2}. It is known from \cite[Theorem 2.1]{HZ12} that \eqref{eq1.2} admits a unique positive $\omega$-periodic solution $z^*$ satisfying the equation \eqref{eq1.7}, if and only if $(1-\rho)a-\rho\delta>0$, where $z^*\in C^1((i\omega,(i+\rho)\omega])\cap C^1(((i+\rho)\omega,(i+1)\omega])$ is bounded. Moreover, if $(1-\rho)a-\rho\delta\leq 0$, then the solution $z(t;z_0)$ of \eqref{eq1.2} converges to $0$ for all $z_0\in [0,+\infty)$ as $t\to+\infty$, while if $(1-\rho)a-\rho\delta>0$, then $\lim_{n\to\infty}(z(t+n\omega;z_0)-z^*(t))=0$ in $C([0,\omega])$ for all $z_0\in(0,+\infty)$. Using this fact, we can prove Theorem \ref{thm1.2}.
\begin{proof}[Proof of Theorem \ref{thm1.2}]
Since $(1-\rho)a-\rho\delta>0$, by Proposition \ref{prop3.6}, there holds that
\[ \lim_{l_2-l_1\to +\infty}\big[ (1-\rho)\sigma_1+\rho\delta \big]=\rho\delta-(1-\rho)a<0, \]
and thus there exists a large $\hat\ell>0$ such that
\[ (1-\rho)\sigma_1+\rho\delta<0\ \ \mathrm{as}\ l_2-l_1>\hat\ell. \]
The existence and uniqueness of $u_{(l_1,l_2)}^*(t,x)$ follow from Theorem \ref{thm1.1}. Note that $z^*(t)$ satisfies
\[ \begin{cases}
z_t=-\delta z, & 0< t\leq \rho\omega, \\
z_t=z(a-bz), & \rho\omega<t\leq \omega, \\
z(0)=z(\omega).
\end{cases} \]
Then $z^*(t)=e^{-\delta t}z^*(0)$ for $0\leq t\leq \rho\omega$. Meanwhile, $u^*_{(l_1,l_2)}(t,x)=e^{-\delta t}u^*_{(l_1,l_2)}(0,x)$ for $0\leq t\leq \rho\omega$ and $l_1\leq x\leq l_2$.

We make an assertion that for each $0<\epsilon\ll 1$, there exists $\ell_{\epsilon}\geq \hat\ell>0$ such that for each $l_1\in(-\infty,-\ell_{\epsilon})$ and $l_2\in(\ell_{\epsilon},+\infty)$,
\begin{equation}\label{eq3.16}
z^*(t)-\epsilon\leq u^*_{(l_1,l_2)}(t,x)\leq z^*(t)+\epsilon,\ \ (t,x)\in[\rho\omega,\omega]\times [l_1,l_2].
\end{equation}
Only the proof for the lower bound will be given here since that for the upper bound is similar. Clearly, $\min_{[0,\omega]}z^*(t)>0$. In fact, if there is $t_0>0$ such that $z^*(t_0)=0$, then $z^*(t)=0$ for all $t\geq t_0$, which is impossible as $z^*$ is periodic in $t$. Set $0<\epsilon\ll 1$. Then there exists $\eta(\epsilon)>0$ such that
\[ \hat z(t):=(1-\eta)z^*(t)\geq z^*(t)-\epsilon>0,\ \ t\in[\rho\omega,\omega]. \]
Observe that
\begin{align*}
&\hat z_t(t)-d\bigg[ \int_{l_1}^{l_2}J(x-y)\hat z(t)\mathrm dy-\hat z(t)\bigg]-\hat z(t)\big[a-b\hat z(t)\big] \\
&\phantom{=}=(1-\eta)z^*_t(t)-(1-\eta)z^*(t)d\bigg[ \int_{l_1}^{l_2}J(x-y)\mathrm dy-1\bigg] \\
&\phantom{==\ }-(1-\eta)z^*(t)\big[a-bz^*(t)\big]-\hat z(t)\big[a-b\hat z(t)\big]+(1-\eta)z^*(t)\big[a-bz^*(t)\big] \\
&\phantom{=}=-d(1-\eta)z^*(t)\bigg[ \int_{l_1}^{l_2}J(x-y)\mathrm dy-1\bigg]-b\eta(1-\eta){z^*}^2(t),\ \ \rho\omega<t\leq \omega,\ l_1\leq x\leq l_2.
\end{align*}
Denote
\[ E_l(t,x):=d(1-\eta)z^*(t)\bigg[ \int_{-l}^{l}J(x-y)\mathrm dy-1\bigg],\ \ (t,x)\in(\rho\omega,\omega]\times \mathbb R.  \]
Since $J\in C(\mathbb R)\cap L^{\infty}(\mathbb R)$ is nonnegative and $\lim\limits_{l\to+\infty}\int_{-l}^{l}J(x)\mathrm dx=1$, we know that $E_l$ is non-decreasing with respect to $l>0$ and continuous, bounded for all $(l,t,x)\in(0,+\infty)\times(\rho\omega,\omega]\times \mathbb R$. It then follows from Dini's Theorem that $E_l(t,x)$ converges to zero locally uniformly in $(\rho\omega,\omega]\times \mathbb R$ as $l\to +\infty$. Hence, there exists $\ell_{\epsilon}\geq \hat\ell>0$ such that for each $l_1\in(-\infty,-\ell_{\epsilon}),l_2\in(\ell_{\epsilon},+\infty)$, the following inequality holds
\begin{equation}\label{eq3.17}
\hat z(t)-d\bigg[ \int_{l_1}^{l_2}J(x-y)\hat z(t)\mathrm dy-\hat z(t)\bigg]-\hat z(t)\big[a-b\hat z(t)\big]< 0,\ \ (t,x)\in (\rho\omega,\omega]\times [l_1,l_2].
\end{equation}

It suffices to prove that for each $l_1\in(-\infty,-\ell_{\epsilon}),l_2\in(\ell_{\epsilon},+\infty)$, we have $\hat z(t)\leq u^*_{(l_1,l_2)}(t,x)$ in $[\rho\omega,\omega]\times [l_1,l_2]$. To this end, we fix any $l_1\in(-\infty,-\ell_{\epsilon}),l_2\in(\ell_{\epsilon},+\infty)$ and set
\[ \beta_*=\inf\Big\{ \beta>0:\hat z(t)\leq \beta u^*_{(l_1,l_2)}(t,x)\ \mathrm{for\ all}\ (t,x)\in(\rho\omega,\omega)\times [l_1,l_2] \Big\}. \]
We see that $\beta_*$ is well-defined and positive since $\min\limits_{[\rho\omega,\omega]\times [l_1,l_2]}u_{(l_1,l_2)}^*(t,x)>0$ and $\hat z(t)$ is bounded. It follows from the continuity of $u_{(l_1,l_2)}^*$ and $\hat z$ that $\hat z(t)\leq \beta_* u^*_{(l_1,l_2)}(t,x)$ for all $(t,x)\in(\rho\omega,\omega)\times [l_1,l_2]$. In particular, there must exist $(t_0,x_0)\in (\rho\omega,\omega)\times [l_1,l_2]$ such that $\hat z(t_0)= \beta_* u^*_{(l_1,l_2)}(t_0,x_0)$.

When $\beta_*\leq 1$, the lower bound in \eqref{eq3.16} holds immediately. On the contrary, suppose that $\beta_*>1$. Let $w(t,x)=\hat z(t)-\beta_* u^*_{(l_1,l_2)}(t,x)$. Then by \eqref{eq3.17} and the equation satisfied by $u^*_{(l_1,l_2)}(t,x)$, a simple calculation yields that
\[ w_t(t,x)<d\bigg[ \int_{l_1}^{l_2}J(x-y)w(t,y)\mathrm dy-w(t,x)\bigg]+\hat z(t)\big[a-b\hat z(t)\big]-\beta_*u^*_{(l_1,l_2)}(t,x)[a-bu^*_{(l_1,l_2)}(t,x)] \]
for $(t,x)\in(\rho\omega,\omega)\times [l_1,l_2]$. However, by the definition of $\beta_*$, we have that $w_t(t_0,x_0)=0$. This together with $\int_{l_1}^{l_2}J(x_0-y)w(t_0,y)\mathrm dy-w(t_0,x_0)\leq 0$ leads to that
\begin{align*}
0&=w_t(t_0,x_0)  \\
&\leq \hat z(t_0)\big[a-b\hat z(t_0)\big]-\beta_*u^*_{(l_1,l_2)}(t_0,x_0)[a-bu^*_{(l_1,l_2)}(t_0,x_0)]  \\
&<\hat z(t_0)\big[a-b\hat z(t_0)\big]-\beta_*u^*_{(l_1,l_2)}(t_0,x_0)[a-b\beta_*u^*_{(l_1,l_2)}(t_0,x_0)]=0,
\end{align*}
which is a contradiction. Consequently, $\beta_*\leq 1$ and so $\hat z(t)\leq u^*_{(l_1,l_2)}(t,x)$ for all $(t,x)\in(\rho\omega,\omega)\times [l_1,l_2]$. In fact, the domain $(\rho\omega,\omega)\times [l_1,l_2]$ can be extended to $[\rho\omega,\omega]\times [l_1,l_2]$ since $\hat z(t)$ and $u^*_{(l_1,l_2)}(t,x)$ are both continuous and bounded. Hence, \eqref{eq3.16} holds true and so $\lim_{-\l_1,l_2\to+\infty}u_{(l_1,l_2)}^*(t,x)=z^*(t)$ in $C_{\mathrm{loc}}([\rho\omega,\omega]\times\mathbb R)$.

On the other hand, when $t\in[0,\rho\omega]$, it holds that $z^*(t)=e^{-\delta t}z^*(0)=e^{-\delta t}z^*(\omega)$ and $u^*_{(l_1,l_2)}(t,x)=e^{-\delta t}u^*_{(l_1,l_2)}(0,x)=e^{-\delta t}u^*_{(l_1,l_2)}(\omega,x)$ for $0\leq t\leq \rho\omega$ and $l_1\leq x\leq l_2$. This means that $\lim\limits_{-l_1,l_2\to+\infty}u_{(l_1,l_2)}^*(t,x)=z^*(t)$ in $C_{\mathrm{loc}}([0,\rho\omega]\times\mathbb R)$. As a result, $\lim\limits_{-l_1,l_2\to+\infty}u_{(l_1,l_2)}^*(t,x)=z^*(t)$ in $C_{\mathrm{loc}}([0,\omega]\times\mathbb R)$. The proof is completed.
\end{proof}

\section{Simulations}
\setcounter{equation}{0}

In this section, we present the simulations to illustrate some of our results. Referring to \cite{VL97}, we choose the form of $J$ to be a simple Laplace kernel:
\[ J(x)=\frac{1}{2D}e^{-\frac{|x|}{D}}\quad \mbox{with $D=20$.} \]

Consider the following parameter sets:
\begin{enumerate}[(P1)]
\item $\delta=0.2,d=0.6,a=1.2,b=0.6,\rho=0.6,\omega=1$;
\item $\delta=0.2,d=1,a=1.2,b=0.6,\rho=0.6,\omega=1$;
\item $\delta=0.8,d=0.6,a=1.2,b=0.6,\rho=0.6,\omega=1$;
\end{enumerate}
and the initial condition
\begin{enumerate}[(IC)]
\item $u_0(x)=\cos(\dfrac{\pi x}{l}),x\in(-l,l)$.
\end{enumerate}

\begin{figure}[!ht]
\centering
\includegraphics[height=5.5cm,width=7.5cm]{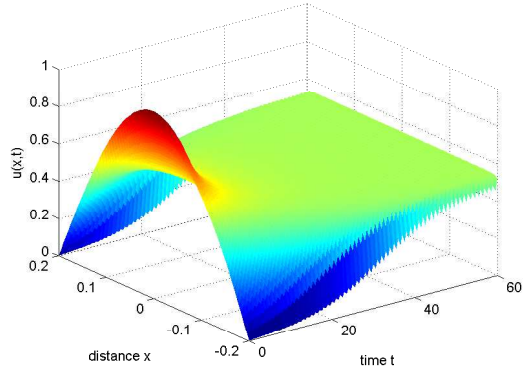}
\caption{Numerical simulations of \eqref{eq1.1} with parameter set (P1) and initial condition (IC), where $l=0.2$.}
\label{fig1}
\end{figure}

\begin{figure}[!ht]
\centering
\includegraphics[height=5.5cm,width=7.5cm]{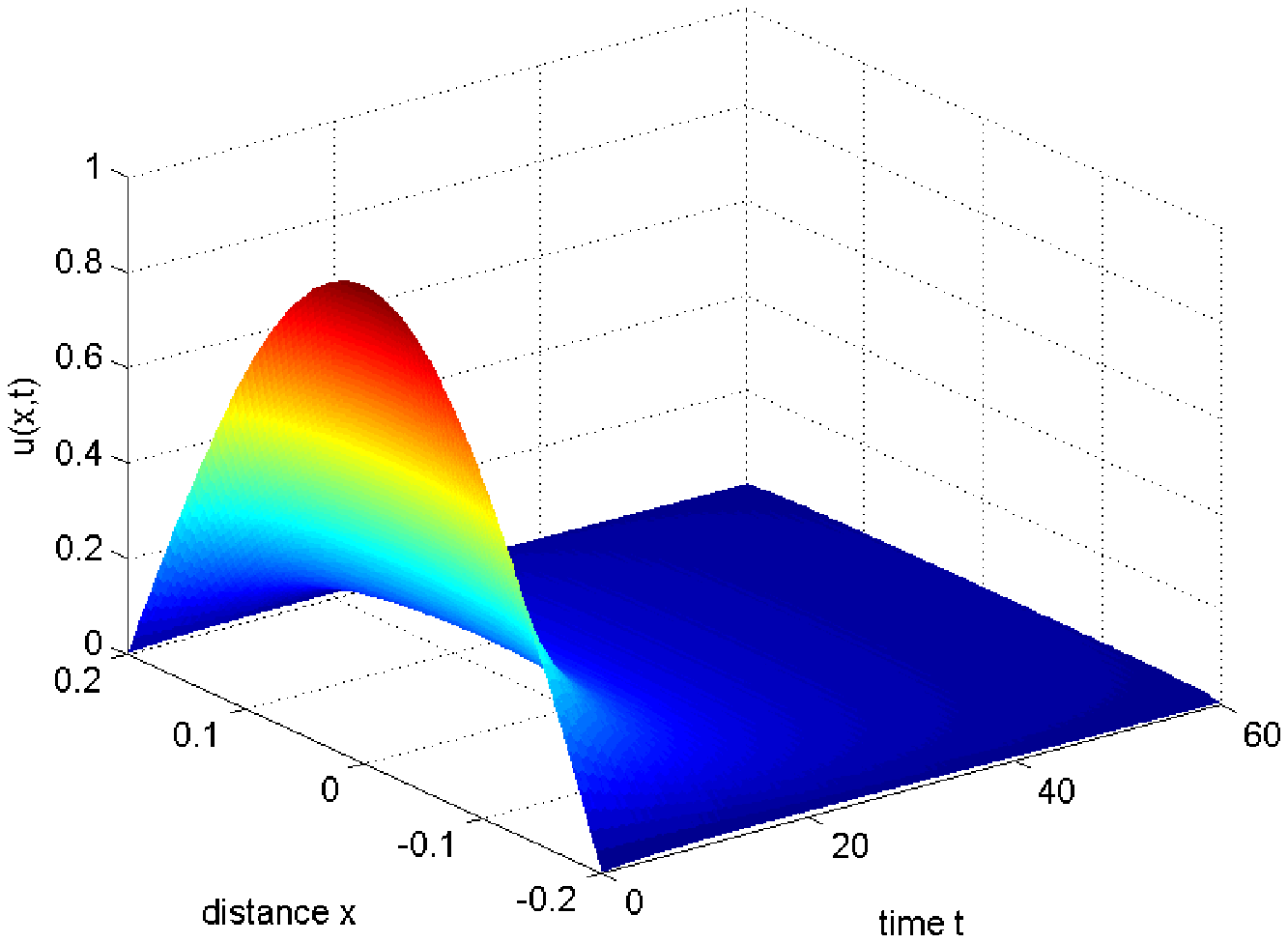}
\includegraphics[height=5.5cm,width=7.5cm]{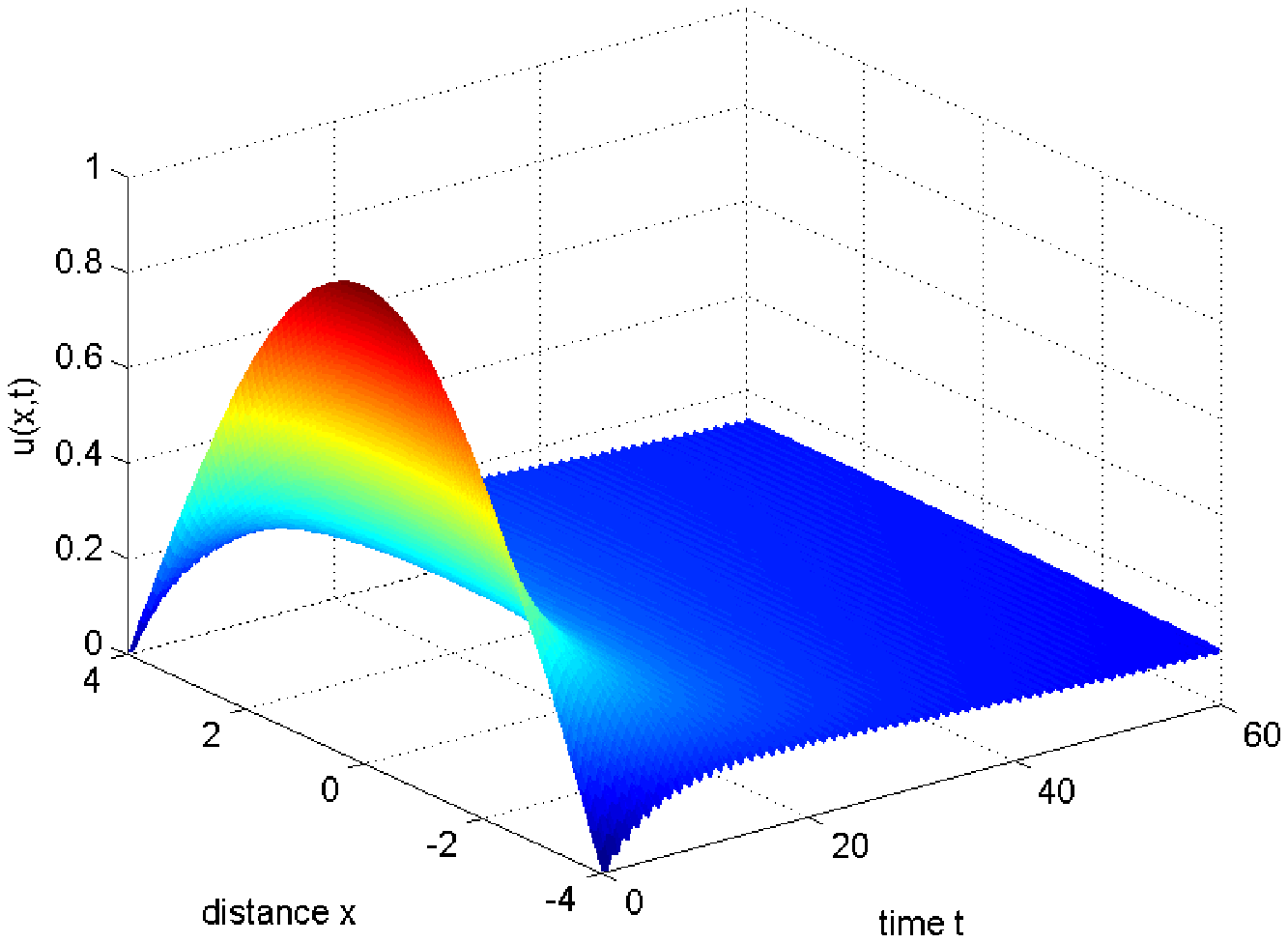}
\caption{Numerical simulations of \eqref{eq1.1} with parameter set (P2) and initial condition (IC). Left: $l=0.2$; Right: $l=4$.}
\label{fig2}
\end{figure}

\begin{figure}[!ht]
\centering
\includegraphics[height=5.5cm,width=7.5cm]{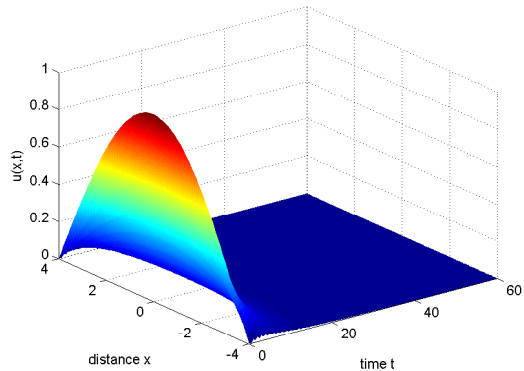}
\caption{Numerical simulations of \eqref{eq1.1} with parameter set (P3) and initial condition (IC), where $l=4$.}
\label{fig3}
\end{figure}

Clearly, the parameter set (P1) satisfies the condition in Theorem \ref{thm1.1} (1). Then Figure \ref{fig1} shows that when the domain length $L:=2l=0.4$, the solution of \eqref{eq1.1} satisfying (P1) and (IC) converges to a spatially nonhomogeneous positive periodic solution. This is consistent with the conclusion of Theorem \ref{thm1.1} (1).

The parameter set (P2) satisfies the condition in Theorem \ref{thm1.1} (2). Then Figure \ref{fig2} shows that when the domain length $L=2l=0.4$, the solution of \eqref{eq1.1} satisfying (P2) and (IC) converges to a spatially nonhomogeneous positive periodic solution, but when $L=2l=8$, the solution of \eqref{eq1.1} with the same parameters and initial condition converges to zero. This is consistent with the conclusion of Theorem \ref{thm1.1} (2).

The parameter set (P3) satisfies the condition in Theorem \ref{thm1.1} (3). Then Figure \ref{fig3} shows that when the domain length $L=2l=8$, the solution of \eqref{eq1.1} satisfying (P3) and (IC) converges to zero. This is consistent with the conclusion of Theorem \ref{thm1.1} (3).

\section{Discussion}
\setcounter{equation}{0}

In this paper, we mainly examine a nonlocal dispersal logistic model with seasonal succession subject to Dirichlet type boundary condition. In Section \ref{sec3}, in order to study the long time behavior of the solutions to \eqref{eq1.1}, we establish a method of time periodic upper-lower solutions, and show that the sign of the eigenvalue $(1-\rho)\sigma_1+\rho\delta$ of the linearized operator can completely determine the asymptotic behavior of the solutions to \eqref{eq1.1}. Meanwhile, we see that the $\omega$-periodic positive solution corresponding to the nonlocal dispersal model \eqref{eq1.1} behaves like the $\omega$-periodic positive solution corresponding to the ODE model \eqref{eq1.2} when the range of the habitat tends to the entire space $\mathbb R$.

In the following, we give some remarks on a nonlocal dispersal logistic model under Neumann type boundary condition, which is associated with model \eqref{eq1.1}, that is,
\begin{equation}\label{eq4.1}
\begin{cases}
u_t=-\delta u, & i\omega<t\leq (i+\rho)\omega,\ l_1\leq x\leq l_2, \\
\displaystyle u_t= d\int_{l_1}^{l_2}J(x-y)\big(u(t,y)-u(t,x)\big)\mathrm dy+u(a-bu), & (i+\rho)\omega<t\leq (i+1)\omega,\ l_1\leq x\leq l_2, \\
u(0,x)=u_0(x)\geq 0, & x\in[l_1,l_2],
\end{cases}
\end{equation}
The kernel function $J:\mathbb R\to \mathbb R$ is assumed to satisfy ${\bf (J)}$. The integral operator $\int_{l_1}^{l_2}J(x-y)\big( u(t,y)-u(t,x) \big)\mathrm dy$ describes diffusion processes, where $\int_{l_1}^{l_2}J(x-y)u(t,y)\mathrm dy$ is the rate at which individuals are arriving at position $x$ from all other places and $\int_{l_1}^{l_2}J(x-y)u(t,x)\mathrm dy$ is the rate at which they are leaving location $x$ to travel to all other sites. Since diffusion takes places only in $[l_1,l_2]$ and individuals may not enter or leave the domain $[l_1,l_2]$, we call it Neumann type boundary condition.

Linearizing system \eqref{eq4.1} at $u=0$, we obtain the time-periodic operator:
\begin{equation}\label{eq4.2}
\tilde L_{(l_1,l_2)}[v](t,x)=\begin{cases}
-v_t-\delta v, & t\in (i\omega,(i+\rho)\omega],\ x\in[l_1,l_2], \\
\displaystyle-v_t+d\int_{l_1}^{l_2}J(x-y)\big(u(t,y)-u(t,x)\big)\mathrm dy+av, & t\in ((i+\rho)\omega,(i+1)\omega],\ x\in[l_1,l_2].
\end{cases}
\end{equation}
A easy calculation yields that $\lambda_1=\delta\rho-a(1-\rho) $
is a eigenvalue of $-\tilde L_{(l_1,l_2)}$ with a positive eigenfunction. Moreover, one can also derive as in Theorem \ref{thm1.1} that
\begin{theorem}\label{thm5.1}
Assume that ${\bf (J)}$ holds and $-\infty<l_1<l_2<+\infty$. Let $u(t,\cdot;u_0)$ be the unique solution to \eqref{eq4.2} with the initial value $u_0(x)\in C([l_1,l_2])$, where $u_0(x)$ is bounded, nonnegative and not identically zero. The following statements are true:
\begin{enumerate}[(1)]
\item If $\delta\rho-a(1-\rho)<0$, then $\lim\limits_{n\to\infty}u(t+n\omega,x;u_0)=z^*(t)$ in $C([0,\omega]\times [l_1,l_2])$, where $z^*(t)$ is the unique positive $\omega$-solution to \eqref{eq1.7};
\item If $\delta\rho-a(1-\rho)\geq 0$, then $\lim\limits_{t\to\infty}u(t,x;u_0)=0$ uniformly for $x\in[l_1,l_2]$.
\end{enumerate}
\end{theorem}
Form above discussion, we can conclude that in spatial homogeneous environment, the nonlocal diffusion model with seasonal succession under Neumann boundary condition has the same dynamical behavior as that of the corresponding ODE model. When the environment is spatially dependent (e.g., the parameter $a$ in model \eqref{eq1.1} or \eqref{eq4.1} is replaced by a spatially dependent function $a(x)$), one can also obtain the existence of a eigenvalue of the linearized operator associated with a positive eigenfunction under the additional compact support condition on the kernel function $J(x)$. Likewise, the sign of such obtained eigenvalue can completely determine the dynamical behavior of model \eqref{eq1.1} or \eqref{eq4.1}. However, in spatially heterogeneous environment, the criteria governing persistence and extinction of the species becomes more difficult to be obtained (see \cite{SV19,SLLY20} for details). We slao remark that if the term $a-bu$ in \eqref{eq1.1} (resp. \eqref{eq4.1}) is replaced by a continuous and strictly decreasing function $f(u)$ on $[0,+\infty)$ satisfying $f(K)\leq 0$ for some $K>0$, then the conclusions showed in Theorems \ref{thm1.1}-\ref{thm1.2} (resp. Theorems \eqref{thm5.1}) still hold with $a=f(0)$.

\section*{Acknowledgements}
\addcontentsline{toc}{section}{Acknowledgements}

This work is supported by the National Natural Science Foundation of China (No. 11871475) and the Fundamental Research Funds for the Central Universities of Central South University (No. 2020zzts040).

\end{document}